\documentclass[10pt]{amsart}
\pdfoutput=1
\usepackage{amsmath}
\usepackage{amsfonts}
\usepackage{amssymb}
\usepackage{amsthm, enumerate, enumitem}
\usepackage{hyperref}
\usepackage{comment}
\usepackage{float}
\usepackage{youngtab}
\usepackage{ytableau}
\usepackage{rotating}
\usepackage{subcaption}
\usepackage{xcolor}
\usepackage{bbding}
\usepackage{tikz}
\usepackage{tikz-cd}
\usepackage[paper=a4paper, margin=3.5cm]{geometry}

\def\RR{{\NZQ R}}
\def\CC{{\NZQ C}}

\def\PP{{\NZQ P}}

\def\frk{\frak}               

\def\Phi{{\frk n}}
\def\Phi{{\frk N}}
\def\TT{{\mathbb T}}

\def\opn#1#2{\def#1{\operatorname{#2}}} 
\opn\chara{char} \opn\length{\ell} \opn\pd{pd} \opn\rk{rk}
\opn\projdim{proj\,dim} \opn\injdim{inj\,dim} \opn\rank{rank}
\opn\spn{span}\opn\Seg{Seg}
\opn\depth{depth} \opn\grade{grade} \opn\height{height}
\opn\embdim{emb\,dim} \opn\codim{codim}

\opn\Tr{Tr} \opn\bigrank{big\,rank}
\opn\superheight{superheight}\opn\lcm{lcm}
\opn\trdeg{tr\,deg}
\opn\reg{reg} \opn\lreg{lreg} \opn\ini{in} \opn\lpd{lpd}
\opn\size{size}\opn\bigsize{bigsize}
\opn\cosize{cosize}\opn\bigcosize{bigcosize}
\opn\sdepth{sdepth}\opn\sreg{sreg}
\opn\link{link}\opn\fdepth{fdepth} \opn\trdeg{trdeg} \opn\mod{mod}
\opn\spann{span}
\opn\div{div} \opn\Div{Div} \opn\cl{cl} \opn\Cl{Cl}
\opn\Spec{Spec} \opn\Supp{Supp} \opn\supp{supp} \opn\Sing{Sing}
\opn\Ass{Ass} \opn\Min{Min}\opn\Mon{Mon} \opn\dstab{dstab} \opn\astab{astab}
\opn\Syz{Syz}
\opn\Ann{Ann} \opn\Rad{Rad} \opn\Soc{Soc} \opn\Aut{Aut}
\opn\Im{Im} \opn\Ker{Ker} \opn\Coker{Coker} \opn\Am{Am}
\opn\Hom{Hom} \opn\Tor{Tor} \opn\Ext{Ext} \opn\End{End}
\opn\Aut{Aut} \opn\id{id}

\opn\nat{nat}
\opn\pff{pf}
\opn\Pf{Pf} \opn\GL{GL} \opn\SL{SL} \opn\mod{mod} \opn\ord{ord}
\opn\Gin{Gin} \opn\Hilb{Hilb}\opn\sort{sort}
\opn\S{S} \opn\dim{dim} \opn\supp{supp}\opn\trdeg{trdeg}\opn\sort{sort}
\opn\aff{aff} \opn\con{conv} \opn\relint{relint} \opn\st{st}
\opn\lk{lk} \opn\cn{cn} \opn\core{core} \opn\vol{vol}
\opn\link{link} \opn\star{star}\opn\lex{lex}
\opn\conv{conv} \opn\Ehr{Ehr}\opn\Pic{Pic}
\opn\Conv{Conv}
\opn\Id{Id}
\opn\gr{gr}

\def\pot#1#2{#1[\kern-0.28ex[#2]\kern-0.28ex]}

\opn\dirlim{\underrightarrow{\lim}}
\opn\inivlim{\underleftarrow{\lim}}
\def\Implies{\ifmmode\Longrightarrow \else
        \unskip${}\Longrightarrow{}$\ignorespaces\fi}
\def\implies{\ifmmode\Rightarrow \else
        \unskip${}\Rightarrow{}$\ignorespaces\fi}
\def\iff{\ifmmode\Longleftrightarrow \else
        \unskip${}\Longleftrightarrow{}$\ignorespaces\fi}

\let\:=\colon

\newtheorem{Theorem}{Theorem}[section]
 \newtheorem{Lemma}[Theorem]{Lemma}
 
 \newtheorem{Proposition}[Theorem]{Proposition}

 \newtheorem{Definition}[Theorem]{Definition}

\def\CC{\mathbb{C}}

\def\RR{\mathbb{R}}

\def\PP{\mathbb{P}}

\title[Proof ]{Proof of a Conjecture of Drton, Sturmfels and Sullivant on the maximum likelihood degree of the Gaussian graphical model of a cycle}
 \author{Rodica Andreea Dinu}
\address{%
	University of Konstanz, Fachbereich Mathematik und Statistik, Fach D 197 D-78457 Konstanz, Germany, and Institute of Mathematics ``Simion Stoilow" of the Romanian Academy, Calea Grivitei 21, 010702, Bucharest, Romania}
	\email{rodica.dinu@uni-konstanz.de}
 
\author{Martin Vodi\v{c}ka}
\address{Šafárik University, Faculty of Science, Jesenná 5, 04154 Košice, Slovakia}
	\email{martin.vodicka@upjs.sk}

\begin{document}

\maketitle

\begin{abstract}
In this article, we compute the precise value of the maximum likelihood degree of the Gaussian graphical model of a cycle, confirming a conjecture due to Drton, Sturmfels and Sullivant.
\end{abstract}

\section{Introduction}

Maximum likelihood estimation (MLE) is a statistical technique that uses certain observable data to estimate the parameters of an assumed probability distribution. This is accomplished by optimizing a likelihood function to make the observed data the most likely under the presumptive statistical model.  The point in the parameter space that maximizes the likelihood function is the maximum likelihood estimate.  The maximum likelihood degree (ML-degree) measures the complexity of finding the MLE. The study of the geometry of the ML-degree of an algebraic statistical model was introduced by Catanese, Ho\c sten, Kethan, and Sturmfels, see \cite{CHKS, HKS}, and it became one of the most active topics in algebraic statistics is the study of the ML-degree, see e.g. \cite{carlos, huh, huh2, orlando, bc, bsp, Seth, Bodensee, mateusz1}.

{\it A multivariate Gaussian distribution} in $\RR^n$ is determined by the mean vector $\mu \in \RR^n$ and a positive definite $n\times n$ matrix $\Sigma$ called the \emph{covariance matrix}. Its inverse matrix $K=\Sigma^{-1}$ is also positive definite and known as the \emph{concentration matrix} of the distribution. Let us fix a linear space $\Lambda$ of symmetric $n\times n$ matrices that contains a positive-definite matrix. As introduced by Anderson, the \textit{linear concentration model} associated to $L$ is the variety
$\Lambda_L:=\{K^{-1}:K\in \Lambda\}$. We identify $S^2(\CC^n)$ with the space of symmetric $n\times n$ matrices over $\CC$ and, for a vector space $V$, we denote by $\PP(V)$ its projectivization.  There are two important invariants of this model. One of them is the degree of the variety, which is the degree of the variety $\Lambda^{-1}:=\overline{\Lambda_L}$, and the other one is the ML-degree of the model, which can be defined as the number of critical points of the log-likelihood function $\ell$ over $S^2(\CC^n)$:
\begin{equation*}\label{eq:lik}
    \ell(\Sigma) = \log \det (\Sigma^{-1}) - \mathrm{tr}(S\Sigma^{-1}),
\end{equation*}
where $S$ is the sample covariance matrix.
Let us consider the projections: $\pi : S^2(\CC^n) \rightarrow S^2/{\Lambda^{\perp}}$.
\begin{Theorem}(\cite{barndorff}, see also \cite{bc})
For a linear concentration model $\Lambda$ and the sample covariance matrix $S\in \Lambda_L$, the MLE is given by the unique matrix $\Sigma_0 \in \Lambda_L$ such that $\pi(\Sigma_0)=\pi(S)$.
\end{Theorem}

\begin{Definition}(\cite{Kathlen})
The ML-degree of a linear concentration model is the degree of the dominant rational map $\pi: \PP(\overline{\Lambda_L}) \dashrightarrow \PP(S^2(\CC^n)/{\Lambda^{\perp}})$.
\end{Definition}
It is known that the degree of the model is always upper bounded by the ML-degree, and equality holds, by a theorem of Teissier \cite[II.2.1.3]{teissier} for general $\Lambda$. 
We will be interested in a special class of linear concentration models, for which the degree and the ML-degree of the model are different.
The {\it Gaussian graphical model} plays a special role in algebraic statistics and, in this case, the linear space $\Lambda\subset S^2(\CC^n)$ is defined by the vanishing of some entries of the matrix.  More precisely, let $\Gamma$ be an (undirected) graph with the vertex set $[n]$ and edges $E(\Gamma)$. Then the linear space $L_{\Gamma}$ associated to this model consists of symmetric matrices $A$ for which $A_{ij}=0$ when $i\neq j$ and $(i,j)\not\in E(\Gamma)$.

The undirected Gaussian graphical model associated with the graph $\Gamma$ is the family of multivariate normal distributions with covariance matrix $\Sigma$ such that $(\Sigma^{-1})_{ij}=0$ for every missing edge $(i,j)$. 
The maximum likelihood estimation for this model is a matrix completion problem, see~\cite[Theorem 2.1.14]{drtonSS}.

So far, the precise value of the ML-degree for Gaussian graphical models was only known in the case of chordal graphs, where the ML-degree is equal to one \cite{bc}.
This makes the cycle $C_n$ the simplest class of graphs for which the ML-degree is not known. Thus, in this article, we will consider only this special case $\Gamma=C_n$. The space $L_{C_n}$ was heavily studied in the literature.  For example, the degree of the variety $L_{C_n}^{-1}$ was computed in \cite[Theorem 1.4]{DMV}, by resolving a conjecture of Sturmfels and Uhler \cite{bc}, and this is given by the following formula:
\[
\deg(L_{C_n}^{-1})=\frac{n+2}{4}{2n \choose n}-3\cdot 2^{2n-3}.
\]
The generators of the ideal of the variety $L_{C_n}^{-1}$ were computed in \cite{conner2023sullivant}, confirming another conjecture of Sturmfels and Uhler. Regarding the ML-degree, we provide a proof of the following theorem which was a long-standing conjecture posed by Drton, Sturmfels and Sullivant:

\begin{Theorem}\label{ML}(\cite[Conjecture in Section 7.4]{drtonSS}) For the $n$-cycle, the maximum likelihood degree of the model is equal to
 $$(n-3)\cdot 2^{n-2}+1.$$
\end{Theorem}
In general, the ML-degree of a linear space $L$ can be computed as the number of intersection points in the intersection $L^{-1}\cap(S+L^{\perp})$ for a generic matrix $S$. Our strategy is to choose a specific matrix $S$, namely $S=\Id$, show that it behaves as a general matrix and compute the intersection. The second part has already been done in \cite[Lemma 3.11]{MLbound}, where the intersection points from $L_{C_n}^{-1}\cap(\Id+L_{C_n}^{\perp})$ are completely described up to the action of a suitable group. However, this result, together with \cite[Theorem 3.15]{MLbound}, gives only a lower bound of the ML-degree. We will continue with analyzing this intersection to prove the precise value of the ML-degree. 

This paper is structured as follows. 
In Section~\ref{identity} we show that the fiber over the class of the identity in the graph of the projection $\pi: S^2(\CC^n)\dashrightarrow S^2(\CC^n)/L_{C_n}^{\perp}$ does not contain any point from the base locus $L_{C_n}^{\perp}$, in Theorem~\ref{noidentity}. This is done by finding a suitable bihomogeneous equation of the ideal of the graph of the projection $\pi$ restricted to $L_{C_n}^{-1}$ which allows us to show that, in the graph, there is no point of the form $(A, \Id)$ with $A$ in the base locus.

Section~\ref{smoothness} is devoted to proving that the intersection $L_{C_n}^{-1}\cap(\Id+L_{C_n}^{\perp})$ is transverse. Since the set of isolated points in the intersection is known \cite[Lemma 3.11]{MLbound}, we have to prove the smoothness of all of them. These points are of many different types, and for each type, we compute the tangent spaces and show that they are zero-dimensional.

We conclude the article by putting it all together to obtain a proof of the desired Theorem~\ref{ML}.

\section{Fiber over the identity at a glance}\label{identity}
\textbf{Setting:} Consider the projection $\pi:\PP(S^2(\CC^n))\dashrightarrow\PP(S^2(\CC^n)/L_{C_n}^\perp)$. Consider the graph of this projection restricted to $L_{C_n}^{-1}$, i.e.

$$\Gamma_n=\overline{\{(A,\pi(A)): A\in L_{C_n}^{-1}\setminus L_{C_n}^\perp\}}\subset \PP(S^2(\CC^n))\times\PP(S^2(\CC^n)/L_{C_n}^\perp).$$

\begin{Theorem}\label{noidentity}
In $\Gamma_n$ there does not exist a point of the form $(A,\Id)$ for $A\in L_{C_n}^\perp$. In other words, in the fiber over (the class of) the identity there is no matrix from $L_{C_n}^\perp$.
\end{Theorem}

We will denote the coordinates in $\PP(S^2(\CC^n))$ by $x_{i,j}$ for $1\le i\le j\le n$. Sometimes, for convenience, we will not be strict about $i\le j$ and we naturally define $x_{j,i}=x_{i,j}$. Analogously, we will denote the coordinates in $\PP(S^2(\CC^n)/L_{C_n}^\perp)$ by $y_{i,j}$ where $1\le i \le j \le i+1 \le n$ or $(i,j)=(1,n)$.

We will often work with submatrices. For $n\times n$ matrix $M$, and sets $I,J\subset\{1,2,\dots,n\}$, we denote by $M(I,J)$ the $|I|\times |J|$ submatrix of $M$ formed by rows indexed by $I$ and columns indexed by $J$.

\begin{Lemma}
Let $I_n$ be the bihomogeneous ideal of $\Gamma_n$. Then for any $3\le k\le n-1$, we have

$$R_k:=y_{1,1}y_{k,k+1}x_{2,k}x_{2,k+1}-y_{1,1}y_{k,k}x_{2,k+1}x_{2,k+1}
-y_{1,2}y_{k,k+1}x_{1,k}x_{2,k+1}+y_{1,2}y_{k,k}x_{1,k+1}x_{2,k+1}-$$
$$
y_{1,2}y_{k+1,k+1}x_{1,k}x_{2,k}+y_{1,2}y_{k,k+1}x_{1,k}x_{2,k+1}
+y_{2,2}y_{k+1,k+1}x_{1,k}x_{1,k}-y_{2,2}y_{k,k+1}x_{1,k}x_{1,k+1}
\in I_n.$$

\end{Lemma}

\begin{proof}\label{quartic-equation}
Let us denote by $M$ the $n\times n$ symmetric matrix  from $\PP(S^2(\CC^n))$ with coordinates $x_{i,j}$. Consider the minors
$\det(M(\{1,2,k\},\{1,k,k+1\}))$ and $\det(M(\{2,k,k+1\},\{1,2,k+1\}))$. These minors are polynomials that are in the ideal of $L_{C_n}^{-1}$, by \cite{conner2023sullivant}, and therefore, are also in the ideal $I_n$.

Thus, $$R'_k:=x_{2,k+1}\cdot \det(M(\{1,2,k\},\{1,k,k+1\}))- x_{1,k}\cdot \det(M(\{2,k,k+1\},\{1,2,k+1\}))\in I_n.$$

By computation we obtain 
$$R'_k=x_{2,k+1}\cdot\det\begin{pmatrix}
x_{1,1}&x_{1,k}&x_{1,k+1}\\
x_{1,2}&x_{2,k}&x_{2,k+1}\\
x_{1,k}&x_{k,k}&x_{k,k+1}\\
\end{pmatrix}-x_{1,k}\cdot \det\begin{pmatrix}
x_{1,2}&x_{2,2}&x_{2,k+1}\\
x_{1,k}&x_{2,k}&x_{k,k+1}\\
x_{1,k+1}&x_{2,k+1}&x_{k+1,k+1}\\
\end{pmatrix}=$$
$$=x_{1,1}x_{k,k+1}x_{2,k}x_{2,k+1}-x_{1,1}x_{k,k}x_{2,k+1}x_{2,k+1}
-x_{1,2}x_{k,k+1}x_{1,k}x_{2,k+1}+x_{1,2}x_{k,k}x_{1,k+1}x_{2,k+1}-$$
$$x_{1,2}x_{k+1,k+1}x_{1,k}x_{2,k}+x_{1,2}x_{k,k+1}x_{1,k}x_{2,k+1}
+x_{2,2}x_{k+1,k+1}x_{1,k}x_{1,k}-x_{2,2}x_{k,k+1}x_{1,k}x_{1,k+1}.$$

Note that two terms $$x_{2,k+1}\cdot (-x_{1,k} x_{2,k}x_{1,k+1}+x_{1,k}x_{1,k}x_{2,k+1})$$ from the first determinant canceled out with the same terms from the second determinant.

Consider any bihomogeneous equation that belongs to the ideal $I_n$. This equation holds for all $(A,\pi(A))$ for $A\in L_{C_n}^{-1}$. If in every term we replace one $x_{i,j}$ by $y_{i,j}$, we obtain an equation that also holds for all $(A,\pi(A))$ for $A\in L_{C_n}^{-1}$. Thus, the resulting equation will also be a bihomogeneous equation from the ideal $I_n$.

We can easily see that if we do this twice (i.e. by replacing two $x_{i,j}$ by the corresponding $y_{i,j}$) for the polynomial $R'_k\in I_n$, we obtain the polynomial $R_k$ from the statement of the lemma.
\end{proof}

\begin{Lemma}\label{2x2-is-good}
Consider a matrix $A\in L_{C_n}^{-1}\cap L_{C_n}^\perp$. Suppose that there exist $1\le i,k\le n$ such that $x_{i,k}\neq 0$ and $x_{i,k+1}=x_{i+1,k}=x_{i+1,k+1}=0$ (indices are taken modulo $n)$. Then $(A,\Id)\not\in \Gamma_n$.
\end{Lemma}
\begin{proof}
Without loss of generality, we may assume $i=1$. If not, we can make a cyclic shift to get to this situation. Since $A\in L_{C_n}^\perp$, we have $x_{1,1}=x_{1,2}=x_{1,n}=0$. Thus, $3\le k\le n-1$.

We can use the equation $R_k\in I_n$ which holds for any point $(A,B)\in \Gamma_n$. After plugging in $x_{1,1}=x_{1,2}=x_{1,n}=0$, there will be only the term $y_{2,2}y_{k+1,k+1}x_{1,k}^2$ left.

Since $x_{1,k}\neq 0$, we must have $y_{2,2}y_{k+1,k+1}=0$, which implies that $B\neq\Id$.
\end{proof}

\begin{Lemma}\label{existence-of-2x2}
Let $A\in L_{C_n}^{-1}\cap L_{C_n}^\perp$ be a matrix. Then there exist indices $1\le i,k\le n$, such that $x_{i,k}\neq 0$ and $x_{i,k+1}=x_{i+1,k}=x_{i+1,k+1}=0$ (indices are taken modulo $n)$.  
\end{Lemma}

\begin{proof}
We are working in the projective space, thus $A\neq 0$. We find the non-zero entry $x_{i,j}$ of the matrix $A$ with the smallest difference $j-i$, where the difference is taken modulo $n$. In other words, we find the non-zero entry that is ``closest" to the main diagonal. Without loss of generality, we may assume $i=1$. 

Consider a submatrix $A(\{1,2\dots,j\},\{j,j+1,\dots,n,1\})$. Since $A\in L_{C_n}^{-1}$, the rank of this matrix is at most 2 (all of its $3\times 3$ minors are 0). However, we have $x_{1,1}=x_{1,2}=x_{1,3}=\dots=x_{1,j-1}=0$, because $x_{1,j}$ was the nonzero entry which is closest to the main diagonal. Thus, the last column (indexed by $1$) of this submatrix has only one non-zero element. Analogously, in the first column (indexed by $j$) only the first element is non-zero. 

This means that the only option for $A(\{1,2\dots,j\},\{j,j+1,\dots,n,1\})$ to be of rank 2 is that all other rows, except the first one and the last one, are 0.

Moreover, $x_{l,l'}=0$ also for any $2\le l,l'\le j-1$, because $x_{1,j}$ is the non-zero entry which is the closest to the main diagonal. Thus, the rows from 2 to $j-1$ of the matrix $A$ are all zero. 

Consider the smallest index $j'>j$ such that $x_{1,j'}=0$. Such index exists, since clearly $x_{1,n}=0$. Now it is sufficient to choose $i=1$, $k=j'-1$ to find the indices $i,k$ from the statement of the lemma.
\end{proof}

\begin{proof}[Proof of Theorem \ref{noidentity}]
Consider any matrix $A\in L_{C_n}^{-1}\cap L_{C_n}^\perp$. By Lemma \ref{existence-of-2x2}, there exist indices $1\le i,k\le n$, such that $x_{i,k}\neq 0$ and $x_{i,k+1}=x_{i+1,k}=x_{i+1,k+1}=0$ (indices are taken modulo $n)$. However, by Lemma \ref{2x2-is-good} this means that the point $(A,\Id)\not\in\Gamma_n$. This concludes the theorem. 
\end{proof}

\section{Smoothness}\label{smoothness}

In this section, we will show that the isolated points in the intersection $L_{C_n}^{-1}\cap(\Id+L_{C_n}^\perp)$ are smooth, and therefore, they are reduced. This will imply that the ML-degree of $L_{C_n}$ is the number of points in the intersection.

In general, we will show the smoothness of an isolated point by proving that the Jacobian matrix evaluated at the point is full-rank, i.e. of rank $\binom {n+1}{2}$.

The ideal of $L_{C_n}^{-1}\cap(\Id+L_{C_n}^\perp)$ is generated by specific $3\times 3$ minors, by \cite[Theorem 1.1]{conner2023sullivant}, and the equations $x_{i,i+1}=0$, $x_{i,i}=1$ for $1\le i\le n$.

Obviously, the partial derivations of the equations $x_{i,i+1}=0$ and $x_{i,i}=1$ are vectors $e_{i,i+1}$ and $e_{i,i}$. Thus, to show the smoothness of a point it is sufficient to show that the submatrix of the Jacobian matrix formed by partial derivations of $3\times 3$ minors by variables $x_{i,j}$ where $2\le |i-j|\le n-2$ is of full-rank, i.e. of rank $\binom {n+1}{2}-2n$.

Let us denote by $\TT(A)$ the vector space generated by the rows of a Jacobian matrix at the point $A\in L_{C_n}^{-1}\cap(\Id+L_{C_n}^\perp)$. Thus, we want to prove that $\TT(A)=S^2(\CC^n)$.

We introduce here the group actions we will use. Consider the set $\mathcal D_n^{\pm}$ of all $n\times n$ diagonal matrices whose diagonal entries are equal to $\pm 1$. This is a group under multiplication and for any $D\in\mathcal D_n^\pm$, we have $D=D^{-1}$. This group acts on $S^2(\CC^n)$ by conjugation. Clearly, the spaces $L_{C_n}$ and $L_{C_n}^{\perp}$ are invariant with respect to this group action. In addition, the set $L_{C_n}^{-1}$ is also invariant. Thus, the intersection $L_{C_n}^{-1}\cap (\Id+L_{C_n}^\perp)$ is also invariant.  

Let us consider another group action, namely the action by cyclic shift. For this, we consider the following matrices:

$$N_n^+:=\begin{pmatrix}
0&1&0&\dots&0&0&0\\
0&0&1&\dots&0&0&0\\
0&0&0&\dots&0&0&0\\
\vdots&\vdots&\vdots&\ddots&\vdots&\vdots&\vdots\\
0&0&0&\dots&0&1&0\\
0&0&0&\dots&0&0&1\\
1&0&0&\dots&0&0&0
\end{pmatrix}, \
N_n^-:=\begin{pmatrix}
0&1&0&\dots&0&0&0\\
0&0&1&\dots&0&0&0\\
0&0&0&\dots&0&0&0\\
\vdots&\vdots&\vdots&\ddots&\vdots&\vdots&\vdots\\
0&0&0&\dots&0&1&0\\
0&0&0&\dots&0&0&1\\
-1&0&0&\dots&0&0&0
\end{pmatrix}.
$$

Note that $(N_n^+)^{-1}=(N_n^+)^T$ and $(N_n^-)^{-1}=(N_n^-)^T$. Thus, (the groups generated by) these matrices act on $S^2(\CC^n)$ by conjugation and $L_{C_n}$ is invariant with respect to that action.
We say that a matrix $A$ is \textit{$N_n^+$-invariant} if $N_n^+A(N_n^+)^{-1}=A$. Analogously, $A$ is \textit{$N_n^-$-invariant} if $N_n^-A(N_n^-)^{-1}=A$. 

Let us denote by $M_n(x),M_n^+(x),M_n^-(x)$ the following matrices:

$$M_n(x):=\begin{pmatrix}
1&x&0&0&\dots&0&0\\
x&1&x&0&\dots&0&0\\
0&x&1&x&\dots&0&0\\
0&0&x&1&\dots&0&0\\
\vdots&\vdots&\vdots&\vdots&\ddots&\vdots&\vdots\\
0&0&0&0&\dots&1&x\\
0&0&0&0&\dots&x&1
\end{pmatrix},$$

$$M_n^+(x):=\begin{pmatrix}
1&x&0&0&\dots&0&x\\
x&1&x&0&\dots&0&0\\
0&x&1&x&\dots&0&0\\
0&0&x&1&\dots&0&0\\
\vdots&\vdots&\vdots&\vdots&\ddots&\vdots&\vdots\\
0&0&0&0&\dots&1&x\\
x&0&0&0&\dots&x&1
\end{pmatrix},\ M_n^-(x):=\begin{pmatrix}
1&x&0&0&\dots&0&-x\\
x&1&x&0&\dots&0&0\\
0&x&1&x&\dots&0&0\\
0&0&x&1&\dots&0&0\\
\vdots&\vdots&\vdots&\vdots&\ddots&\vdots&\vdots\\
0&0&0&0&\dots&1&x\\
-x&0&0&0&\dots&x&1
\end{pmatrix}.$$

We denote $P_n(x):=\det(M_n(x))$. 

The points in the intersection $L_{C_n}^{-1}\cap(\Id+L_{C_n}^\perp)$,  were characterized in \cite[Lemma 3.11]{MLbound}:

\begin{Lemma}\label{odd-intersection}
We have the following equalities:
$$L_{C_{2n+1}}^{-1}\cap(\Id+L_{C_{2n+1}}^{\perp})=\{\Id\}\cup \{DM_{2n+1}^+(x)^{-1}D : D\in\mathcal D_{2n+1}^\pm;P_{n-1}(x)+xP_{n-2}(x)=0\}.$$
$$L_{C_{2n}}^{-1}\cap(\Id+L_{C_{2n}}^{\perp})=\{\Id\}\cup \{DM_{2n}^+(x)^{-1}D : D\in\mathcal D_{2n}^\pm;P_{n-1}(x)-x^2P_{n-3}(x)=0\}\cup $$
$$\cup \{DM_{2n}^-(x)^{-1}D : D\in\mathcal D_{2n}^\pm;P_{n-2}(x)=0\}\cup \{D\mathcal{C}_{2n}D : D\in\mathcal{D}_{2n+1}^\pm\},$$
where
$\mathcal{C}_n$ is the checkerboard matrix, i.e. $$(\mathcal C_n)_{i,j}=\begin{cases}
 0,&\text{ if } i+j \text{ is odd}  \\
 1,&\text{ if } i+j \text{ is even}.
\end{cases}$$

\end{Lemma}

We can see that there are only a few types of points up to the action of the group $\mathcal D^+_n$. We will look at each of them separately.

\subsection{The identity}
\begin{Proposition}\label{id}
The $\Id$ is smooth in the intersection $L_{C_{n}}^{-1}\cap(\Id+L_{C_{n}}^\perp)$.
\end{Proposition}
\begin{proof}
For $4\le k\le n$ we consider the following $3\times 3$ minor from the ideal of $L_{C_n}^{-1}$: $$\delta(1,2,3)(1,3,k):=\det\begin{pmatrix}
x_{1,1}&x_{1,3}&x_{1,k}\\
x_{1,2}&x_{2,3}&x_{2,k}\\
x_{1,3}&x_{3,3}&x_{3,k}\\
\end{pmatrix}$$

If we consider the partial derivations of this minor and evaluate at the identity, it is clear that $\frac{\partial(\delta(1,2,3)(1,3,k)))}{\partial x_{2,k}}=-x_{1,1}x_{3,3}$ evaluated at the identity is -1 and all the other partial derivations evaluated at the identity are 0. 

Thus, we see that by taking partial derivations of $\delta(1,2,3)(1,3,k)$ we obtain a row of the Jacobian matrix which is equal to $-e_{2,k}$. By taking a cyclic shift, we obtain rows $-e_{i,j}$ for any $i,j$ with $2\le |i-j|\le n-2$. Therefore, the Jacobian matrix at the identity has full rank.
\end{proof}
\subsection{The checkerboard matrices}

\begin{Proposition}\label{checkerboardmatr}
The points of the form $D\mathcal C_nD$ for any $D\in \mathcal D^{\pm}_{n}$ are smooth in the intersection  $L_{C_{n}}^{-1}\cap(\Id+L_{C_{n}}^\perp)$ for everyeven $n$.
\end{Proposition}
\begin{proof}
Since the map given by $A\mapsto DAD$ for any matrix $D\in \mathcal D^{\pm}_{n}$ is regular and it preserves the intersection $L_{C_n}^{-1}\cap(\Id+L_{C_n}^\perp)$, it is sufficient to prove smoothness of a matrix $\mathcal C_{n}$. This will imply the smoothness of any matrix of the form $D\mathcal C_nD$ for $D\in \mathcal D^{\pm}_{n}$.

For any odd $3\le k\le n-1$, we consider the following $3\times 3$ minor from the ideal of $L_{C_n}^{-1}$:$$\delta(1,2,3)(1,4,k):=\det\begin{pmatrix}
x_{1,1}&x_{1,4}&x_{1,k}\\
x_{1,2}&x_{2,4}&x_{2,k}\\
x_{1,3}&x_{3,4}&x_{3,k}\\
\end{pmatrix}$$

Since this $3\times 3$ matrix evaluated at $\mathcal C_{n}$ is $\begin{pmatrix}
1&0&1\\
0&1&0\\
1&0&1\\
\end{pmatrix}$,
 we can easily see that only the partial derivations by the variables $x_{1,1},x_{1,k},x_{1,3},x_{3,k}$ are non-zero. More precisely, 

 $$\frac{\partial(\delta(1,2,3)(1,4,k))}{\delta x_{1,1}}\bigg\rvert_{\mathcal C_n}=-\frac{\partial(\delta(1,2,3)(1,4,k))}{\delta x_{1,k}}\bigg\rvert_{\mathcal C_n}=$$
 $$
 =\frac{\partial(\delta(1,2,3)(1,4,k))}{\delta x_{1,3}}\bigg\rvert_{\mathcal C_n}=-
 \frac{\partial(\delta(1,2,3)(1,4,k))}{\delta x_{3,k}}\bigg\rvert_{\mathcal C_n}=1.$$

This means that by taking partial derivation of this minor we get a row of the Jacobian matrix which is $-e_{1,k}-e_{1,3}+e_{3,k}$. Note that we do not consider the partial derivation by $x_{1,1}$ since we already disregarded all the variables $x_{i,i}$ and $x_{i,i+1}$.

For $k=3$ we get a row $-2e_{1,3}$, and by cyclic shift, we obtain a row $-2e_{i,i+2}$ for any $1\le i\le n$. These rows generate the vector space generated by vectors $e_{i,i+2}$. Thus $$\langle e_{i,i+2},\  1\le i\le n \rangle\subset \TT(\mathcal C_n).$$

Now, we proceed by induction on $j$ and we will show that $$\langle e_{i,i+2j},\  1\le i\le n \rangle\subset \TT(\mathcal C_n).$$ We simply put $k=2j+1$ to obtain a row $-e_{1,2j+1}-e_{1,3}+e_{3,2j+1}$ of the Jacobian matrix. Since, by the induction hypothesis, $e_{1,3},e_{3,2j+1}\in \TT(\mathcal C_n)$, we have that $e_{1,2j+1}\in \TT(\mathcal C_n)$. By cyclic shift, we prove the statement.

Consider a minor $\delta(1,2k+1,n)(2,3,2k)$ for $(n-2)/2\ge k\ge 2$. We do the same trick - we compute the partial derivations, evaluate at the point $\mathcal C_n$ and get a vector from $\TT(\mathcal C_n)$:

$$\delta(1,2k+1,n)(2,3,2k)=\det\begin{pmatrix}
x_{1,2}&x_{1,3}&x_{1,2k}\\
x_{2,2k+1}&x_{3,2k+1}&x_{2k,2k+1}\\
x_{2,n}&x_{3,n}&x_{2k,n}\\
\end{pmatrix}.$$

After evaluating at $\mathcal C_{n}$, we obtain the matrix $\begin{pmatrix}
0&1&0\\
0&1&0\\
1&0&1\\
\end{pmatrix}$.

We can see that the only partial derivations that are non-zero are the partial derivations by $x_{1,2}, x_{1,2k},x_{2,2k+1},x_{2k,2k+1}$. Since we already know that $e_{1,2},e_{2k,2k+1}\in\TT(\mathcal C_n)$, we obtain that $e_{1,2k}+e_{2,2k+1}\in \TT(\mathcal C_n)$.

If we perform a cyclic shift in this argument, we get that also $e_{2,2k+1}+e_{3,2k+2}\in \TT(\mathcal C_n)$. By subtracting the last two vectors, we get $e_{1,2k}-e_{3,2k+2}\in \TT(\mathcal C_n)$.

Next, we consider the minor $\delta(1,2,3)(2k,2k+1,2k+2)$ for $(n-2)/2\ge k\ge 2$. Similarly, we obtain that $e_{1,2k}-e_{1,2k+2}-e_{3,2k}+e_{3,2k+2}\in \TT(\mathcal C_n)$.  

Further, we show by induction on $k$ that $(k-1)e_{1,4}-e_{1,2k}\in \TT(\mathcal C_n)$ for $k\ge 2$. For $k=1$ it holds, because $e_{1,2}\in \TT(\mathcal C_n)$, for $k=2$, it is trivial since the expression is a zero vector.

Assume that $k\ge 2$ and the statement holds for all $k'\le k$.

We have $$e_{1,2k}-e_{1,2k+2}-e_{3,2k}+e_{3,2k+2},e_{1,2k}-e_{3,2k+2},e_{1,2k-2}-e_{3,k}\in \TT(\mathcal C_n).$$
 This yields

$$e_{1,2k}-e_{1,2k+2}-e_{3,2k}+e_{3,2k+2}+e_{1,2k}-e_{3,2k+2}-e_{1,2k-2}+e_{3,2k}=2e_{1,2k}-e_{1,2k-2}-e_{1,2k+2}\in \TT(\mathcal C_n).$$

By induction hypothesis, 
$$(2k-2)e_{1,4}-2e_{1,2k},-(k-2)e_{1,4}+e_{1,2k-2}\in \TT(\mathcal C_n).$$ 
Therefore,
$$2e_{1,2k}-e_{1,2k-2}-e_{1,2k+2}+(2k-2)e_{1,4}-2e_{1,2k}-(k-2)e_{1,4}+e_{1,2k-2}=
ke_{1,4}-e_{1,2k+2}\in\TT(\mathcal C_n).$$ 
 which proves the statement.

Now we plug in $k=(n-1)/2$ to get $((n-1)/2) e_{1,4}-e_{1,n-2}\in \TT(\mathcal C_n)$. In addition, we have $e_{1,n-2}+e_{3,n}\in\TT(\mathcal C_n) $ which implies $((n-1)/2) e_{1,4}-e_{3,n}\in \TT(\mathcal C_n)$. Furthermore, we have $e_{1,4}-e_{2,5}\in \TT(\mathcal C_n)$ and, by cylic shift, we obtain that  $-e_{1,4}+e_{3,n}\in\TT(\mathcal C_n)$.

Finally, by adding $-e_{1,4}+e_{3,n}+((n-1)/2) e_{1,4}-e_{3,n}$, we obtain that $((n+1)/2) e_{1,4}\in \TT(\mathcal C_n)$. This implies $e_{1,2k}\in \TT(\mathcal C_n)$ and, by cyclic shift, $e_{i,j}\in \TT(\mathcal C_n)$ for any $i,j$ with odd difference.

Hence, we obtained that all $e_{i,j}$ belong to $\TT(\mathcal C_n)$, and we can conclude that the point $\mathcal C_n$ is smooth in the intersection.
\end{proof}

\subsection{The other type of matrices}

We start with a lemma, where the indices are taken modulo $n$ as always.

\begin{Lemma}\label{TT}
 For all $j\ge 2$, we have the following inclusions: $$\left\{ \sum_{i=1}^n \alpha_i e_{i,i+j}; \sum_{i=1}^n \alpha_i=0\right\} \subset \TT((M^+_{n}(z))^{-1}),$$
$$\left\{ \sum_{i=1}^n \alpha_i e_{i,i+j}; \sum_{i=1}^{n-j} \alpha_i-\sum_{n-j+1}^n \alpha_i=0\right\} \subset \TT((M^-_{n}(z))^{-1}),
$$

for all $(M^+_{n}(z))^{-1},(M^-_{n}(z))^{-1}\in L_{C_n}^{-1}\cap(\Id+L_{C_n}^\perp)$.
 
\end{Lemma}
\begin{proof}
We start with the case of $(M^+_{n}(z))^{-1}$.
Let us consider the $3\times 3$ minor $$\delta(1,2,3)(1,3,4):=\det\begin{pmatrix}
x_{1,1}&x_{1,3}&x_{1,4}\\
x_{1,2}&x_{2,3}&x_{2,4}\\
x_{1,3}&x_{3,3}&x_{3,4}\\
\end{pmatrix}$$

The proof of \cite[Lemma 2.3]{MLbound} implies that we can assume $x_{i,i+2}=-1$ for all $1\le i\le n$.
The $3\times 3$ matrix evaluated at $(M^+_n(z))^{-1}$ is $\begin{pmatrix}
1&-1&a_{1,4}\\
0&0&-1\\
-1&1&0\\
\end{pmatrix}$,
where we do not know the value of $a_{1,4}$. However, we know that $a_{1,4}\neq 0$, from the proof of \cite[Lemma 2.3]{MLbound} and we also know that the matrix $(M^+_{n}(z))^{-1}$ is $N_n^+$ invariant. This information about the entries of the matrix $(M^+_{n}(z))^{-1}$ will be sufficient for our proof
.
The only partial derivations of this $3\times 3$-minor that are not 0 is $$\frac{\partial(\delta(1,2,3)(1,3,4))}{\delta x_{1,3}}\bigg\rvert_{\mathcal (M^+_n(z))^{-1}}=2.$$

Again, we disregarded the partial derivation by $x_{i,i}$ and $x_{i,i+1}$.
Therefore, we have $e_{1,3}\in \TT( (M^+_n(z))^{-1})$. By cyclic shift, we obtain that $e_{i,i+2}\in\TT( (M^+_n(z))^{-1})$, which implies the statement of the lemma for $j=2$.

Now, we consider the minor $$\delta(1,2,4)(1,4,n)=\det\begin{pmatrix}
x_{1,1}&x_{1,4}&x_{1,n}\\
x_{1,2}&x_{2,4}&x_{2,n}\\
x_{1,4}&x_{4,4}&x_{4,n}\\

\end{pmatrix}.$$

This $3\times 3$-matrix evaluated at $(M^+_n(z))^{-1}$ is $$\begin{pmatrix}
1&a_{1,4}&0\\
0&-1&-1\\
a_{1,4}&1&a_{4,n}\\
\end{pmatrix}.$$ 

Again, we do not know the values $a_{1,4},a_{n,4}$. There are several partial derivations of this minor that are non-zero. However, we may disregard the derivations by $x_{1,1},x_{1,2},x_{2,4},x_{1,n},x_{4,4},x_{2,n}$ since the corresponding vectors $e_{i,j}$ belong to the space $\TT((M^+_n(z))^{-1})$.

Hence, we obtain a vector $-e_{4,n}-2a_{1,4}e_{1,4}\in \TT((M^+_n(z))^{-1})$. Considering an analogous minor, namely $\delta(2,3,n)(3,4,n)$, we obtain a vector $-e_{4,n}-2a_{3,n}x_{3,n}\in \TT((M^+_n(z))^{-1}) $. By subtracting these two vectors and use the fact that $a_{1,4}=a_{3,n}\neq 0$, we obtain that $e_{1,4}-e_{3,n}\in \TT((M^+_n(z))^{-1})$. 

Analogously, by cyclic shift, we obtain that $e_{i,i+3}-e_{i+1,i+4}\in \TT(( M^+_n(z))^{-1})$ which shows the statement of the lemma for $j=3$.
From now on we can proceed by induction:

We consider the minor $\delta(1,2,3)(1,4,1+j)$. By computing the partial derivations, we obtain a vector

$$-a_{2,1+j}e_{1,4}-e_{3,1+j}-a_{1,4}e_{2,j+1}-e_{1,j+1}\in \TT((M^+_n(z))^{-1}).$$

Analogously, by considering the minor $\delta(2,3,4),(2,5,2+j)$
we obtain
$$-a_{2,1+j}e_{2,5}-e_{3,2+j}-a_{1,4}e_{3,j+2}-e_{2,j+2}\in \TT((M^+_n(z))^{-1}).$$

By taking the difference, we get
$$a_{2,1+j}(e_{1,4}-e_{2,5})+(e_{3,i+j}-e_{4,2+j})+a_{1,4}(e_{2,j+1}-e_{3,2+j})+(e_{1,j+1}-e_{2,j+2})\in \TT((M^+_n(z))^{-1}).$$

The differences $e_{1,4}-e_{2,5},e_{3,i+j}-e_{4,2+j},e_{2,j+1}-e_{3,2+j}$ are in $\TT((M^+_n(z))^{-1})$ by the induction hypothesis, thus we get $e_{1,j+1}-e_{2,j+2}\in \TT((M^+_n(z))^{-1})$ as well which concludes the first part of the lemma.

The proof is completely analogous in the case of the matrix $(M^-_n(z))^{-1}$. We consider the same minors; the only difference is that there will be a change of signs when one of the indices is greater than $n$. 
\end{proof}

\begin{Proposition}\label{uglymatr}
The points of the form $\{DM_{2m+1}^+(z)^{-1}D$ for $D\in\mathcal D_{2m+1}^\pm;P_{m-1}(z)+zP_{m-2}(z)=0\}$ are smooth in the intersection $L_{C_{2m+1}}^{-1}\cap(\Id+L_{C_{2m+1}}^\perp)$.

The points of the form $\{DM_{2m}^+(x)^{-1}D$ for $D\in\mathcal D_{2m}^\pm;P_{m-1}(x)-x^2P_{m-3}(x)=0$ and 
$ \{DM_{2m}^-(x)^{-1}D$ for $D\in\mathcal D_{2m}^\pm;P_{m-2}(x)=0\}$ are smooth in the intersection 
$L_{C_{2m}}^{-1}\cap(\Id+L_{C_{2m}}^\perp)$.
\end{Proposition}
\begin{proof}
Similarly as in the previous section, it is sufficient to prove the smoothness of the points $ (M^+_n(z))^{-1}$ and $ (M^-_n(z))^{-1}$ for the appropriate values of $z$ (by \cite[Lemma 2.17]{MLbound}). We will start with the point $ (M^+_n(z))^{-1}$, the other case will be analogous.

Suppose that the intersection $L_{C_n}^{-1}\cap(\Id+L_{C_n}^\perp)$ is not smooth at the point $(M^+_n(z))^{-1}$. Both $L_{C_n}^{-1}$ and $(\Id+L_{C_n}^\perp)$ are smooth at this point since it is a regular matrix. This means these two varieties have a common tangent vector at the point $(M^+_n(z))^{-1}$. By Lemma~\ref{TT}, we can conclude that this common tangent vector lies in the space of $N^+_n$-invariant matrices.

Thus, if we denote by $$W_n=\{A\in S^2(\CC^n)\:\ N^+_nA(N^+_n)^{-1}=A\},$$ there exists a line $\ell'\subset W_n$ which contains a point $M^+_n(z)$ and is tangent to $L_{C_n}^{-1}$.

Note that the same thing is true if we pass to the projective space, there is a projective line $\ell\subset\PP(W_n)$ which contains the (class of the) point $M^+_n(z)$ and is tangent to the projectivization of $L_{C_n}^{-1}$.

Let us consider the birational map $$\phi:\PP(S^2(\CC^n))\dashrightarrow \PP(S^2(\CC^n))$$ defined by $\phi(A)=A^{-1}$. This map is an isomorphism on the set of regular matrices. Since the matrix $M^+_n(z)$ is regular, we know that the variety $\phi(L_{C_n}^{-1})$ is tangent to the line $\phi(\ell)$ at the point $M^+_n(z)$. However, at least locally at the set of all regular matrices, we have that $\phi(\PP(L_{C_n}^{-1}))=\PP(L_{C_n})$. Also, note that $\phi(\PP(W_n))=\PP(W_n)$. Thus, after applying the inverse map $\phi$, we obtain two vector spaces with a common tangent vector, i.e. $\phi(\ell)\in \PP(L_{C_n}\cap W_n)$. However, the last intersection is just the line given by the (closure of the) set of matrices $M^+_n(x)$. 

As a consequence, we obtain that the line $\{M^+_n(x)\:\ x\in\CC\}$ is tangent to the variety $\overline{\phi(\Id+L_{C_n}^\perp)}$, i.e their intersection is not smooth.

However, the intersection $\{M^+_n(x)\:\ x\in\CC\}\cap \overline{\phi(\Id+L_{C_n}^\perp)}$ is determined by the equation $x^{n-2}+(-1)^nP_{n-2}(x)$ (this follows from the proof of \cite[Lemma 2.13]{MLbound}).

For odd $n=2m+1$, we have $x^{n-2}+(-1)^nP_{n-2}(x)=(P_m(x)-xP_{m-1}(x))(P_{m-1}(x)+xP_{m-2}(x)).$

Since $z$ is not a root of the first factor, and the second factor, $P_{m-1}(x)+xP_{m-2}(x)$, is a divisor of the polynomial $P_{n-3}(x)$ which has only simple roots by \cite[Corollary 3.14]{MLbound}, it follows that the intersection  $\{M^+_n(x)\:\ x\in\CC\}\cap \overline{\phi(\Id+L_{C_n}^\perp)}$
is smooth at the point $M^+_n(z)$ which is a contradiction.

In the even case $n=2m$, we get to the analogous contradiction with the different factorization
$x^{n-2}+(-1)^nP_{n-2}(x)=P_{m-1}(x)((P_{m-1}(x)-x^2P_{m-3}(x)).$
This shows that the point $(M^+_n(z))^{-1}$ is smooth at the intersection $L_{C_n}^{-1}\cap(\Id+L_{C_n}^\perp)$.

Analogously, we can show that the point $(M^-_n(z))^{-1}$ is smooth, the difference will be that instead of $N_n^+$-invariant matrices, we must use the $N_n^-$-invariant matrices.

\end{proof}

\begin{Theorem}\label{smooth-intersection}
All points in the intersection $L_{C_{n}}^{-1}\cap(\Id+L_{C_{n}}^\perp)$ for even $n$ are smooth.
\end{Theorem}
\begin{proof}
In Lemma~\ref{odd-intersection} it is provided the complete list of the intersection points. By using Propositions \ref{id}, \ref{checkerboardmatr}, \ref{uglymatr}, we deduce that all the intersection points are smooth.
\end{proof}

Now we provide the proof for our main theorem:

\begin{proof}[Proof of Theorem~\ref{ML}]
The ML-degree of the space $L_{C_n}$ is equal to the degree of the projection which is the rational map $\Tilde{\pi}: \PP(L_{C_n}^{-1})\dashrightarrow \PP(S^2(\CC^n)/L_{C_n}^\perp)$.

This is equal to the number of intersection points of $L_{C_n}^{-1}\cap (S+L_{C_n}^\perp)$ for a generic matrix $S$. More generally, it is equal to the number of the intersection points for any matrix $S$ providing that the intersection is transverse and the fiber over matrix $S$ in the graph of $\Tilde{\pi}$ does not contain any point from the base locus $L_{C_n}^\perp$. By Theorems \ref{smooth-intersection} and \ref{noidentity}, this is true for $S=\Id$. 

Thus, the ML-degree of $L_{C_n}$ is equal to the number of intersection points of $L_{C_{n}}^{-1}\cap(\Id+L_{C_{n}}^\perp)$. By Lemma~\ref{odd-intersection}, this number is precisely $(n-3)\cdot 2^{n-2}+1$, which proves the theorem.
\end{proof}

\vspace{.2in}
\section*{Acknowledgement}
This paper was written while the first author visited the Institute for Advanced Study, Princeton, in October 2024. She expresses her gratitude to June Huh and Mateusz Micha\l{}ek for the invitation and interesting discussions.
RD was supported by the Alexander von Humboldt Foundation and DFG grant nr 467575307. MV was supported by Slovak VEGA grant 1/0152/22.

\end{document}